\documentclass[10pt]{article}
\usepackage{color}
\usepackage{amssymb}
\usepackage{amsthm,array,amssymb,amscd,amsfonts,latexsym, url}
\usepackage{amsmath}

\newtheorem{theo}{Theorem}[section]
\newtheorem{prop}[theo]{Proposition}

\newtheorem{lemm}[theo]{Lemma}

\newtheorem{rema}[theo]{Remark}

\newtheorem{sublemma}[theo]{Sublemma}
\title{Hyper-K\"ahler compactification of  the intermediate Jacobian fibration of a cubic fourfold : the twisted case}
\author{Claire Voisin}

\date{}

\newfont{\gothic}{eufb10}
\begin{document}
\maketitle
\setcounter{section}{-1}

\begin{abstract}  The starting point of this note  is
 our recent paper with Laza and Sacc\`a on the construction of deformations of O'Grady's $10$-dimensional manifolds as compactifications of intermediate Jacobian fibrations
 associated to cubic fourfolds. The note  provides   a complement to that paper consisting in the analogous construction
in the twisted case, leading to isogenous but presumably not isomorphic or birational hyper-K\"ahler manifolds.
 \end{abstract}
\section{Introduction}

Hyper-K\"ahler geometry is a geometry of a very restricted  type which is part of
the more general setting of $K$-trivial compact K\"ahler geometry.
The existence of hyper-K\"ahler manifolds rests
on Yau's theorem \cite{yau}. Hyper-K\"ahler manifolds are complex
manifolds of even complex dimension $2n$ with a Ricci-flat K\"ahler metric and parallel everywhere nondegenerate
holomorphic $2$-form. If one forgets about the metric, the complex manifolds one obtains
can always be deformed to projective complex manifolds.  Hodge theory plays  a major role in the deformation theory of these complex manifolds. In fact they are not only locally but also globally
determined determined by their period point (see \cite{beauville}, \cite{verbi}), namely the de Rham cohomology class of the closed holomorphic $2$-form.
It is also remarkable that studying the period map for these manifolds
led Beauville and Bogomolov to the discovery of the so-called
Beauville-Bogomolov quadratic form, whose existence  is their most striking topological property.
The situation concerning the construction and classification of
deformation types of hyper-K\"ahler manifolds is very strange:
Two infinite series are known (see \cite{beauville}), each one having one type
for  each even dimension, and furthermore two sporadic (families of) examples in dimension $6$ and $10$ were constructed by O'Grady (\cite{OG1}, \cite{OG2}). Another strange feature of the theory is the following: the simplest kyper-K\"ahler manifolds are
$K3$ surfaces, with particular examples constructed as Kummer surfaces, hence associated with
abelian surfaces or $2$-dimensional complex tori. There are many different ways of
associating to $K3 $ surfaces or abelian surfaces higher dimensional  hyper-K\"ahler manifolds with $b_2=23$ (and also $24$ for the O'Grady examples) built as (desingularizations of) moduli spaces
of simple sheaves on them. Unfortunately algebraic  $K3$ surfaces have only $19$ parameters,
 while these algebraic hyper-K\"ahler manifolds have their deformation spaces (as polarized manifolds) of dimension
$b_2-3>19$, so the general one does not come from an (algebraic) $K3$ surface.
Very curiously, cubic hypersurfaces in $\mathbb{P}^5$, which have $20$ parameters, have also been very much used
for the construction of various $20$ parameters families of algebraic hyper-K\"ahler
manifolds. This is well-understood and even expected Hodge-theoretically, but rather unexpected geometrically. In fact
the variation of Hodge structure on the cohomology of degree $4$
of a cubic fourfold exactly looks like the variation of Hodge structure on the degree
$2$ cohomology of a polarized  hyper-K\"ahler manifold with $b_2=23$.
We will describe several instances of these constructions in Section \ref{sec2}.
The most recent such construction has been provided in \cite{lazasaccavoisin} and we will
achieve   in Sections \ref{sec3} and \ref{sec4} a twisted variant of that construction.
Let $X\subset \mathbb{P}^5 $ be a smooth cubic fourfold. Let $U\subset B:=(\mathbb{P}^5)^*$
be the open set parametrizing smooth hyperplane sections $Y\subset X$.
The family of intermediate Jacobians $J(Y_t)_{t\in U}$ is a
smooth projective fibration $\pi_U:\mathcal{J}_U\rightarrow U$ which according
 to \cite{donagimarkman} has a nondegenerate closed holomorphic $2$-form
making the fibration  Lagrangian.
The following is the main result of \cite{lazasaccavoisin}:
\begin{theo}\label{theolsv} There exists a flat projective fibration $\pi:\overline{\mathcal{J}}\rightarrow B$
extending $\pi_U$, such that the total space $\overline{\mathcal{J}}$ is smooth and hyper-K\"ahler. Furthermore, $\overline{\mathcal{J}}$ is a deformation of a
 $10$-dimensional O'Grady hyper-K\"ahler manifold.
\end{theo}
We can be slightly more precise, introducing the open set $U_1\subset B$ parametrizing at worst $1$-nodal hyperplane sections
of $X$. The motivation for introducing $U_1$ is the fact that ${\rm codim}\,(B\setminus U_1\subset B)\geq 2$,
and this is a key point in the strategy of \cite{lazasaccavoisin}. The family of intermediate Jacobians
has a standard extension $\mathcal{J}_{U_1}\rightarrow U_1$ over $U_1$, first
as a family of quasiabelian schemes $\mathcal{J}_{U_1}^{\circ}\rightarrow U_1$, and then by applying the Mumford compactification: the fibers of $\pi_{U_1}:\mathcal{J}_{U_1}^{\circ}\rightarrow U_1$ are $\mathbb{C}^*$-bundles over
the four-dimensional intermediate Jacobians $J(\widetilde{Y}_t)$ for $t\in U_1\setminus U$,
and the Mumford compactification is obtained by compactifying the
$\mathbb{C}^*$-bundle to a $\mathbb{P}^1$-bundle with the $0$- and $\infty$-sections glued via a translation.
The compactified hyper-K\"ahler manifold $\overline{\mathcal{J}}$  is in fact a compactification of
$\mathcal{J}_{U_1}$.

The intermediate Jacobian fibration $\mathcal{J}_U$ has a twisted version $\mathcal{J}_U^T$ (which appears in \cite{voisinjag} and plays an important role there, although it is not defined very carefully).
 There are several ways of understanding it (see Section \ref{sec3}).  The set of points in the fiber of
 $\mathcal{J}_U^T$ over a point $t\in U$ identifies with the set of $1$-cycles of degree $1$ in the fiber $Y_t$
 modulo rational equivalence (see Section \ref{sec3}).
  We will
construct in Section \ref{sec3} $\mathcal{J}_U^T$ as an algebraic variety
(a torsor over $\mathcal{J}_U$) and a natural extension $\mathcal{J}_{U_1}^T$ of $\mathcal{J}_U^T$ over $U_1$ which is \'etale (or analytically) locally isomorphic
to $\mathcal{J}_{U_1}$ over $U_1$, thus getting a twisted version of $\mathcal{J}_{U_1}$.
Note that $\mathcal{J}_{U_1}^T$ carries a nondegenerate closed holomorphic $2$-form, for exactly the same reasons $\mathcal{J}_{U_1}$ does.
The goal of this note is to prove the following twisted analogue
of Theorem \ref{theolsv}:
\begin{theo}\label{theotwist} Let $X$ be general cubic fourfold. There exists a flat projective fibration $\pi^T:\overline{\mathcal{J}}^T\rightarrow B$
extending $\pi_{U_1}^T:\mathcal{J}_{U_1}^T\rightarrow U_1$, such that the total space $\overline{\mathcal{J}}^T$ is smooth and hyper-K\"ahler.
\end{theo}
 The conclusion above holds for general $X$, that is, all
cubic fourfolds parametrized by a certain Zariski open set of the space of all
smooth cubic fourfolds. When the cubic carries a degree $4$ integral Hodge class which restricts to the generator
of $H^4(Y_t,\mathbb{Z})$ on its hyperplane sections, but it is still general in the sense that it belongs to this Zariski open set, which happens along
a countable dense union of hypersurfaces in the moduli space, the two
fibrations $\mathcal{J}_{U}$ and $\mathcal{J}_{U}^T$ are isomorphic over $U$ and the two
varieties $\overline{\mathcal{J}}^T$ and $\overline{\mathcal{J}}$ are thus birational. It is not clear that they are then isomorphic, but Huybrechts \cite{huybrechts} shows  that they are deformation equivalent. In particular the varieties $\overline{\mathcal{J}}^T$ are deformation equivalent
to OG10 manifolds.
One may wonder if the varieties $\overline{\mathcal{J}}^T$ and $\overline{\mathcal{J}}$ are birational for very general, nonspecial $X$. It is clear that they are not birational as Lagrangian fibrations, since
one has a rational section, while the other does not have  a rational section. It is likely that the varieties we are considering have a unique Lagrangian fibration so in fact are not isomorphic or even birational,
but we have not pursued this.

\vspace{0.5cm}

{\bf Thanks.} {\it I am happy to thank Radu Laza and Giulia Sacc\`a with whom I  had many interesting discussions at the origin of  this
  note.
}
\section{Deforming and constructing hyper-K\"ahler manifolds}
\subsection{Deformation theory and the period map}
The Bogomolov-Tian-Todorov theorem says that a compact K\"ahler manifold
$X$
with trivial canonical bundle has unobstructed deformations. This means that
a first order deformation of the complex structure of $X$, which is given by an element by $H^1(X,T_X)$ called the Kodaira-Spencer class, see \cite[9.1.2]{voisinbook}, extends to a deformation of arbitrarily large order, and
in fact there exists in this case a universal family $\mathcal{X}\rightarrow B$ where $B$ is a
ball in $H^1(X,T_X)\cong \mathbb{C}^N$, $\mathcal{X}$ is a complex manifold
and $\phi$ is proper holomorphic with central fiber
$\mathcal{X}_0\cong X$ such that the Kodaira-Spencer
map $T_{B,0}\rightarrow H^1(X,T_X)$ (the classifying map for first order deformations) is an isomorphism.

Assume now that $X$ is hyper-K\"ahler. It is a general fact that the fibers
$\mathcal{X}_t$ for $t$ small are still hyper-K\"ahler: indeed, the close fibers are still K\"ahler
as small deformations of a compact K\"ahler manifold, and the holomorphic $2$-form still exists on
$\mathcal{X}_t$ for small $t$ because the Hodge numbers $h^{p,q}(X):={\rm dim}\,H^{p,q}(X),\,\,H^{p,q}(X)=H^q(X,\Omega_X^p)$, are constant
under a deformation of compact K\"ahler manifolds.
Ehresmann's fibration theorem tells us that the family $\phi:\mathcal{X}\rightarrow B$ is
$\mathcal{C}^\infty$ trivial and in particular topologically trivial:
$\mathcal{X}\stackrel{homeo}{\cong} \mathcal{X}_0\times B$.
In particular, we have canonical identifications
\begin{eqnarray}
\label{eqisocomp}
H^2(\mathcal{X}_t,\mathbb{C})\cong H^2(\mathcal{X}_0,\mathbb{C})=H^2({X},\mathbb{C}).
\end{eqnarray}
The period map
$\mathcal{P}$ associates to $t\in B$ the class
$[\sigma_t]$ of the closed holomorphic $2$-form $\sigma_t$ on $\mathcal{X}_t$, seen as an element
of $H^2(X,\mathbb{C})$ via the isomorphism (\ref{eqisocomp}).
Note that $\sigma_t$ is defined up to a multiplicative coefficient, hence
$[\sigma_t]$
is well-defined only in $\mathbb{P}(H^2(X,\mathbb{C}))$. Hence the period map takes value in
$\mathbb{P}(H^2(X,\mathbb{C}))$. It is a general result due to Griffiths that the period map
is holomorphic. Furthermore the computation of its differential shows in our case
that the period map is an immersion.
Note that ${\rm dim}\,H^1(X,T_X)={\rm dim}\,H^1(X,\Omega_X)=b_2(X)-2$, as Hodge
theory provides the Hodge decomposition $H^2(X,\mathbb{C})=H^{2,0}(X)\oplus H^{1,1}(X)\oplus H^{0,2}(X)$ with
$H^{1,1}(X)\cong H^1(X,\Omega_X)$, the two other spaces being $1$-dimensional.
It follows that
$\mathcal{P}(B)\subset \mathbb{P}(H^2(X,\mathbb{C}))$ is a germ $\mathcal{D}_{loc}$ of analytic hypersurface.
One can use it (see \cite{beauville}, \cite{huysurvey}) to
construct the Beauville-Bogomolov quadratic form $q$ on
$H^2(X,\mathbb{Q})$.
\begin{theo} (see \cite{beauville}) Let $X$ be a hyper-K\"ahler manifold of dimension
$2n$.
There exists a nondegenerate quadratic form of signature $(3,b_2(X)-3)$ on $H^2(X,\mathbb{Q})$ which is defined up to a multiplicative coefficient by the following condition: for some positive rational number
$\lambda$, one has \begin{eqnarray}\label{eqqbo}\int_X\alpha^{2n}=\lambda q(\alpha)^n
\end{eqnarray}
for any
$\alpha\in H^2(X,\mathbb{Q})$.
\end{theo}
This form is usually normalized in such a way that it takes integral values on $H^2(X,\mathbb{Z})$ and is not divisible as an intersection form on $H^2(X,\mathbb{Z})$;
it is then uniquely determined.
Although Beauville gives an explicit formula for $q$,
the existence of $q$ follows directly from
the study of the period map.
Indeed, we observe that  classes $[\sigma_t]\in \mathcal{D}_{loc}$ satisfy
$[\sigma_t]^{n+1}=0$ in $H^2(X,\mathbb{C})$. Indeed, they are classes of holomorphic
$2$-forms $\sigma_t$ on (some deformation
$X_t$ of) $X$, and clearly $\sigma_t^{n+1}=0$ as a form on $X$.
It follows that the hypersurface $\mathcal{H}_{2n}$ of degree $2n$ in $\mathbb{P}(H^2(X,\mathbb{C}))$
which is defined by the degree $2n$
homogeneous form $f$ given by the formula $f(\alpha)= \int_X\alpha^{2n}$ contains
$\mathcal{D}_{loc}$ and
has multiplicity at least $n$ along it, hence has a component of multiplicity
$\geq n$, which is not  a hyperplane. One then concludes that this component is a quadric hypersurface $Q$ and that
$\mathcal{D}_{loc}$ is open in   $Q$. Thus the equation $q$ of the quadric $Q$ and $f$ are related by
$f=\lambda q^n$ for some coefficient $\lambda$. Finally, using the fact that $f$ is rational, one
sees that both $\lambda$ and $q$ can be taken to be rational.
The statement concerning the signature of $q$ follows from the Hodge index theorem
and further identities derived from (\ref{eqqbo}).

The Verbitsky Torelli theorem for marked hyper-K\"ahler manifolds involves the integral structure on
cohomology. It says the following:
\begin{theo} (\cite{verbi}, \cite{huybrechtsbourbaki}) Let $X,\,X'$ be two hyper-K\"ahler manifolds
which are deformation equivalent. Assume there is an isomorphism $\phi:H^2(X,\mathbb{Z})\cong H^2(X',\mathbb{Z})$ which is obtained by transporting cohomology along a path of deformations from
$X$ to $X'$ and such that $\phi([\sigma_X]=[\sigma_{X'}]$. Then $X$ and $X'$ are birationally
equivalent. If furthermore $\phi$ sends one K\"ahler class on $X$ to a K\"ahler class
on $X'$, $X$ and $X'$ are isomorphic.
\end{theo}
\subsection{Constructing hyper-K\"ahler manifolds \label{sec2}}
One particularity of hyper-K\"ahler geometry is the fact that although
their deformation   theory
has an analytic and transcendental character, all known examples have been constructed by
means of
algebraic geometry. Note again that algebraic geometry provides constructions for the underlying
complex manifolds, but of course not for the hyper-K\"ahler metrics.
In dimension $2$, hyper-K\"ahler manifolds are $K3$ surfaces (see \cite{ast2}). These surfaces are all obtained by deforming smooth quartic surfaces in $\mathbb{P}^3$, defined by one
degree $4$ homogeneous equation $f(X_0,\ldots,X_3)$.
The projective dimension of the space of such polynomials is $34$, while ${\rm Aut}\,\mathbb{P}^3=PGl(4)$ has dimension $15$, so that the family of isomorphism classes of quartic
$K3$ surfaces  has dimension $19$. This is a general fact of
hyper-K\"ahler geometry. The number $h^{1,1}$ in this case is equal to $20$ and this is the dimension
of the space of all deformations of the $K3$ surface $S$. The deformations of $S$ as a quartic surface are restricted since these $K3$ surfaces carry a holomorphic line bundle $L$ such that
$c_1(L)^2=4$. The class $c_1(L)$ is a Hodge class, that is integral of type $(1,1)$, and equivalently
it is integral and satisfies $q(c_1(L),\sigma_S)=0$. Note that the Beauville-Bogomolov form $q$ is in this case the intersection pairing on middle cohomology. The general situation will be similar: given
a hyper-K\"ahler manifold $X$ equipped with an ample line bundle $L$, the polarized deformations of $X$, that is the deformations of the pair $(X,L)$ consisting of a complex manifold and a holomorphic line bundle on it, have $h^{1,1}(X)-1$ parameters, and are parameterized by the period domain $\mathcal{D}_l\subset \mathcal{D}$ determined by the condition
$q(\sigma,l)=0$ where $l=c_1(L)$.

In dimension $4$, there are two known topological types  of hyper-K\"ahler manifolds, namely
$S^{[2]}$ and $K_2(A)$, where $S$ is a $K3$ surface and $A$ is a $2$-dimensional abelian surface (or complex torus).
The second punctual Hilbert scheme $X^{[2]}$ is defined for any algebraic variety or complex manifold
$X$ as the set of length $2$ subschemes of $X$, which consist either
of two distinct points of $X$  or a point with a tangent vector.
If $X$ is smooth, this is the desingularization of the symmetric product
$X^{(2)}$ obtained by blowing-up the diagonal. The generalized Kummer variety
$K_2(A)$ is a fourfold obtained as follows: using the group structure of $A$,  the smooth variety (or complex manifold)  $A^{[3]}$
parametrizing length $3$ subschemes of $A$ admits the sum morphism $A^{[3]}\rightarrow A$, and
$K_2(A)$ is any fiber of this morphism. These two examples appear in \cite{beauville}, as their next generalizations:
It is a general theorem due to Fogarty \cite{fogarty} that the $n$-th Hilbert
scheme $S^{[n]}$ is smooth for any $n$ and any smooth surface $S$. For $n\leq3$, the result is true without any assumptions on the dimension because length $3$ subschemes are supported on smooth surfaces. Beauville shows that
$S^{[n]}$ and $K_n(A)\subset A^{[n+1]}$ are hyper-K\"ahler manifolds of dimension $2n$ for $S$ a
$K3$ surface, $A$ a $2$-dimensional complex torus. The holomorphic $2$-form on
$S^{[n]}$ comes from the holomorphic $2$-form $\sigma_S$ on $S$ by descending the $(2,0)$-form
$\sum_ipr_i^*\sigma_S$ on $S^n$, which is invariant under the action of the symmetric group
$\mathfrak{S}_n$, to the smooth part of the quotient $S^{(n)}=S^n/\mathfrak{S}_n$ and then showing that it extends to a $(2,0)$-form on the
desingularization $S^{[n]}$.

Apart from these two series of examples, only two other deformation types are known, of respective dimensions $6$ and $10$, and they have been constructed by O'Grady \cite{OG1}, \cite{OG2}.
The $10$-dimensional example is constructed as follows: Consider the moduli space of simple
 coherent sheaves $\mathcal{E}$ of rank $2$ on a $K3$ surface $S$, satisfying
 ${\rm det}\,\mathcal{E}=\mathcal{O}_S$ (equivalently $c_1(\mathcal{E})=0$ in $H^2(S,\mathbb{Z})$), and ${\rm deg}\,c_2(\mathcal{E})=4$. By results of Mukai \cite{mukai}, this is a smooth
 algebraic variety with an everywhere  nondegenerate closed  holomorphic (in fact algebraic) $2$-form. The variety is quasiprojective and admits a natural projective completion which is a moduli space of
 semistable sheaves (with respect to a given polarization). The later is singular along the locus
 of nonstable objects (for example $\mathcal{I}_z\oplus \mathcal{I}_{z'}$ where $z$ and $z'$ are
 two subschemes of length $2$ in $S$). O'Grady constructs a hyper-K\"ahler desingularization of this singular moduli space. This variety has $b_2=24$, hence its algebraic deformations have
 $21$ parameters. Those constructed starting from an algebraic $K3$ surface have $19$ parameters.

Cubic fourfolds, that is smooth cubic hypersurfaces in $\mathbb{P}^5$, play an unexpected role
in this study. Although they are Fano varieties, hence have a priori nothing to do  with hyper-K\"ahler varieties, the Hodge decomposition on their degree $4$ cohomology takes the form
$$H^4(X,\mathbb{C})=H^{3,1}(X)\oplus H^{2,2}(X)\oplus H^{1,3}(X),$$
with ${\rm dim}\,H^{3,1}(X)={\rm dim}\,H^{1,3}(X)=1$. Furthermore, the space $H^{2,2}$ has dimension $21$, and it contains one Hodge (in fact algebraic) class, namely $h^2$, $h=c_1(\mathcal{O}_X(1))$. This Hodge structure is polarized by the intersection form
$\langle\,,\,\rangle_X$
on the middle degree cohomology $H^4(X,\mathbb{Z})$ (which will thus play the role of the Beauville-Bogomolov form). This notion of polarization of the Hodge structure says that the line $H^{3,1}(X)\subset H^4(X,\mathbb{C})$ is isotropic for $\langle\,,\,\rangle_X$, and futhermore there is a (open) sign
condition saying in our case that $\langle \alpha,\overline{\alpha}\rangle_X<0$ for $0\not=\alpha\in H^{3,1}(X)$.
Thus up to shift of bigrading by $(1,1)$, (and change of sign for the quadratic
form), what we get is a polarized Hodge structure of hyper-K\"ahler type.
Furthermore, the cubic has $20$ parameters (this is the dimension of the projective space
of cubic homogeneous polynomials in $6$ variables minus the dimension of $PGl(6)$) and in fact
the period map $t\mapsto  H^{3,1}(X_t)\subset H^4(X_t,\mathbb{C})=H^4(X,\mathbb{C})$ is a local isomorphism to the polarized period  domain $\mathcal{D}_{h^2}$ which as before is an open set in the quadric
defined by the Poincar\'e intersection pairing on $H^4(X,\mathbb{Z})^{\perp h^2}$. The cubic is there only to give us a polarized variation of Hodge structure  of hyper-K\"ahler type with $20$ parameters.
Surprisingly enough, there are many hyper-K\"ahler manifolds associated with a cubic fourfold,
which have their variation of Hodge structure isomorphic to the one described above  on the degree $4$ cohomology of the cubic (with a shift of degree).
The known examples are:

(1)  The variety of lines of a cubic fourfold $X$ \cite{bedo}. This is a hyper-K\"ahler fourfold
$F(X)$ and the incidence correspondence $P\subset F(X)\times X$ (which is the universal
$\mathbb{P}^1$-bundle over $F(X)$) induces an isomorphism
$P^*: H^4(X,\mathbb{Z})\cong H^2(F(X),\mathbb{Z})$ of Hodge structures. Beauville and Donagi show that for some special cubic fourfolds (more precisely ``Pfaffian'' cubic fourfolds), $F(X)$ becomes isomorphic to $S^{[2]}$ for some $K3$ surfaces $S$.

(2) The variety of rational curves of degree $3$ in $X$ \cite{llsvs}. This is a $\mathbb{P}^2$-bundle on a variety which is the blow-up of a hyper-K\"ahler $8$-fold $Z(X)$ along a Lagrangian embedding
of $X$ in $Z(X)$. Again, the variations of Hodge structures on $H^4(X)$ and $H^2(Z(X))$ coincide, using the fact that up to a unininteresting summand, the $H^2$ of $Z(X)$ and $F_3(X)$ coincide, and
using the universal correspondence between $F_3(X)$ and $X$. This variety $F(X)$ has been shown
in \cite{addingtonlehn} to be
of the same  deformation type as $S^{[4]}$, where $S$ is a $K3$ surface.

(3)  The intermediate Jacobian fibration of a cubic fourfold has been already mentioned in the introduction. This is a quasiprojective holomorphically symplectic $10$-fold. It is proved in
\cite{lazasaccavoisin} that it admits an hyper-K\"ahler compactification
$\overline{\mathcal{J}}$, and furthermore that $\overline{\mathcal{J}}$ is deformation equivalent to
O'Grady's $10$-dimensional varieties. One remark is that contrarily to the previous cases
where we exhibited a complete family of projective hyper-K\"ahler manifolds, (hence a family with
$b_2-3$ parameters), the expected dimension for the O'Grady examples should be $21$ while our family
has $20$ parameters. This is due to the fact that the varieties that we construct
are by construction Lagrangian fibered over
$\mathbb{P}^5$. Hence their Picard number is at least $2$,  containing one class pulled-back from $\mathbb{P}^5$ and one ample class, and in fact it has to be generally equal to $2$, since the family we construct has $20$ parameters and the manifolds have $b_2=24$.

The rest of this paper is devoted to the description of the twisted version of (3), giving rise to a second $20$ parameters family of  deformations of O'Grady $10$-dimensional varieties.
\begin{rema}{\rm The cubic fourfold is not the only Fano variety
$X$  which has
associated hyper-K\"ahler manifolds whose variation of Hodge structure on $H^2$ is isomorphic (up to a shift of degree) to the variation of Hodge structure on some cohomology group of $X$.
Another example can be found in \cite{debarrevoisin}, where  $X$ is a Pl\"ucker hyperplane section
of the Grassmannian $G(3,V_{10})$ of $3$-dimensional vector subspaces of a given
$10$-dimensional vector space $V_{10}$, so that $X$ defined by a general element $s$ of $\bigwedge^3V_{10}^*$. We show that the fourfold $K_s\subset G(6,V_{10})$ consisting of
$6$-dimensional vector subspaces of $V_{10}$
on which $s$ vanishes identically is a hyper-K\"ahler fourfold, with VHS on $H^2_{prim}$ isomorphic to the VHS on $H^{20}(X)_{prim}$.
}
\end{rema}
\section{The twisted intermediate Jacobian fibration: smooth and $1$-nodal case \label{sec3}}
Let $X\subset \mathbb{P}^5$ be  a smooth cubic fourfold, and $B=(\mathbb{P}^5)^\vee$ be the set of hyperplane sections of $X$ and let $\mathcal{Y}\subset B\times X$ be the universal hypersurface. Let $U\subset U_1\subset B$ be the Zariski open sets of $B$ parametrizing smooth, resp. $1$-nodal, hyperplane sections of $X$. The reason for
introducing $U_1$  is, as in \cite{lazasaccavoisin} which we follow closely, the fact that
$B\setminus U_1$ has codimension $2$ in $B$, so that the flat
fibration with smooth, hence normal, total space $\overline{\mathcal{J}}^T$ will be determined
(see Section \ref{sec4})  by
its restriction  $\mathcal{J}_{U_1}^T$ over $ U_1$.

As we mentioned in the introduction,  the twisted intermediate Jacobian fibration
$\mathcal{J}_{U}^T$ over $U$
can be interpreted as parametrizing $1$-cycles of degree $1$ in the fibers of the universal
(smooth) hypersurface $u:\mathcal{Y}_U\rightarrow U$ over $U$.
If we want to describe it as  an abstract torsor over $\mathcal{J}_U$ or even better
as a complex
manifold, it can be  defined by considering Deligne-Beilinson cohomology along the fibers of  $u:\mathcal{Y}_U\rightarrow U$. This gives an exact sequence of sheaves of groups on $U$
 $$0\rightarrow \mathcal{J}_U\rightarrow \mathcal{H}^4_{\mathcal{D}}\stackrel{c}{\rightarrow} R^4u_*\mathbb{Z}\rightarrow 0,$$
 where the sheaf $R^4u_*\mathbb{Z}$ is canonically isomorphic to $\mathbb{Z}$.
 We can then define
 $\mathcal{J}_U^T$ as $c^{-1}(1)$ (note that we identify here the analytic  group fibration
  and its  sheaf of holomorphic sections).
 The definition above is not good to understand the extension $\mathcal{J}_{U_1}^T$
 and does not describe $\mathcal{J}_U^T$ as an algebraic object.
 We will thus give here an alternative description of $\mathcal{J}_{U_1}^T$ actually as a twist
 of $\mathcal{J}_{U_1}$. Our goal in this section is to establish the following result:
 \begin{prop} \label{proptwist} There exists a quasiprojective variety $\mathcal{J}_{U_1}^T$ with a
 projective morphism $\pi_{U_1}^T:\mathcal{J}_{U_1}^T\rightarrow U_1$  which has the following
 properties.

 (1) The family $\mathcal{J}_{U_1}^T$ is \'etale locally isomorphic to
 $\mathcal{J}_{U_1}$ over $U_1$.

 (2) Let $f: U'\rightarrow U$ be a base change, with $U'$ smooth, and assume that there is
 a codimension $2$ cycle $\mathcal{Z}\in {\rm CH}^2(\mathcal{Y}_{U'})$ that has degree
 $1$ in the fibers of $\mathcal{Y}_{U'}\rightarrow U'$. Then there is a canonical section
 $U'\rightarrow \mathcal{J}_{U'}^T$ of the fibration $\mathcal{J}_{U'}^T :=f^*\mathcal{J}_{U}^T\rightarrow U'$. Equivalently,
 there is a morphism $\Phi_{\mathcal{Z}}:U'\rightarrow \mathcal{J}_{U}^T$ over $U$.
 \end{prop}
 More precisely, the morphism $\Phi_{\mathcal{Z}}$ is in fact compatible with the Abel-Jacobi map in the
 sense that for two points $u',\,u''\in U'$ with $f(u')=f(u'')=u$,
 $\Phi_{\mathcal{Z}}(u')=\Phi_{\mathcal{Z}}(u'')+\Phi_{\mathcal{Y}_u}(\mathcal{Z}_{u'}-\mathcal{Z}_{u''})$,
 where $\Phi_{\mathcal{Y}_u}$ is the Abel-Jacobi map of $\mathcal{Y}_u$. This last fact is automatic,
 due to the universal property  of the Abel-Jacobi map (see \cite{murre}).
 The proof of the proposition will be done through a few lemmas.

 To start with, we  observe that $\mathcal{J}_{U_1}$ contains a Zariski open set $\mathcal{J}_{U_1}^{\circ}$, which is a group scheme over $U_1$ and differs from $\mathcal{J}_{U_1}$  only over $U_1\setminus U$. Over $U_1\setminus U$, the fibers of $\mathcal{J}_{U_1}^{\circ}$ are quasiabelian varieties, and more precisely, the fiber $\mathcal{J}_{U_1,t}^{\circ}$ over
 $t\in U_1\setminus U$ is a $\mathbb{C}^*$-bundle over the intermediate
 Jacobian $J(\widetilde{\mathcal{Y}_t})$, where $\widetilde{\mathcal{Y}_t}$ is the desingularization of ${\mathcal{Y}_t}$ obtained by blowing-up the node.
 Denoting by $o_t$ the singular point of $\mathcal{Y}_t$,
  the fiber $\mathcal{J}_{U_1,t}^{\circ}$ can be understood in terms of algebraic cycles
 as the group of $1$-cycles homologous to $0$ in
 $\mathcal{Y}_t\setminus\{o_t\}$  modulo rational equivalence relative to $o_t$, that it
 modulo ${\rm div}\,\phi$, where $\phi$ is a rational function on $W\subset \mathcal{Y}_t$
 which is well defined and invertible near $o_t$ if $o_t\in W$ (see \cite{refsrivie}, \cite[Lecture 3]{bloch}).
 The fibers $\mathcal{J}_{U_1,t}$ of the compactified Jacobian fibration
 $\mathcal{J}_{U_1}$ are obtained as the Mumford compactification
 \cite{mumfordcomp}, obtained
 by replacing the $\mathbb{C}^*$-bundle mentioned above by the corresponding $\mathbb{P}^1$-bundle with the sections $0$ and $\infty$ glued via a translation
 of the base $J(\widetilde{\mathcal{Y}_t})$. The  translation is given by an element
 $\eta\in J(\widetilde{\mathcal{Y}_t})$ and the $\mathbb{C}^*$-bundle is given by its class
 $\eta'\in {\rm Pic}^0J(\widetilde{\mathcal{Y}_t})$ ; the two classes $\eta$ and $\eta'$ must be identified via the (principal) $\Theta$-divisor of $J(\widetilde{\mathcal{Y}_t})$. This is indeed the condition that this compactification  admits an ample divisor which is the limit of the Theta-divisors on the smooth fibers $J(\mathcal{Y}_s)$. (The reader will find in \cite{bardelli}, \cite{collinobardelli} the interpretation of this class.)
 Notice that the (sheaf of holomorphic sections of) the
analytic  group scheme $\mathcal{J}_{U_1}^{\circ}$ is given by the formula
\begin{eqnarray}\label{eqformulaju10}
\mathcal{J}_{U_1}^{\circ}=\mathcal{H}^{1,2}/R^3u_{1*}\mathcal{Z},
\end{eqnarray}
 where $u_1:\mathcal{Y}_{U_1}\rightarrow U_1$ is the universal hypersurface over $u_1$, (thus, by Picard-Lefschetz theory,  $R^3u_{1*}\mathcal{Z}=j_{1*}H^3_\mathbb{Z}$ is the natural extension of the local system $H^3_\mathbb{Z}=R^3u_{*}\mathcal{Z}$ existing on $U$,) and $\mathcal{H}^{1,2}$ is the Deligne
 extension of the Hodge bundle $$\mathcal{H}^{1,2}_U=R^1u_*\Omega_{\mathcal{Y}_U/U}^2$$ existing on $U$.
 Let us prove the following lemma:
 \begin{lemm} \label{leaction} The group scheme $\mathcal{J}_{U_1}^{\circ}$ acts (over $U_1$) on the compactification
 $\mathcal{J}_{U_1}$.
 \end{lemm}
 \begin{proof} One easy way to prove this is to observe that we have a family $\pi_U: \mathcal{J}_U\rightarrow U$ of principally polarized varieties with a canonical Theta divisor
 $\Theta_U \subset \mathcal{J}_U$ inducing a relative  isomorphism
 $$\mathcal{J}_U\cong {\rm Pic}^0(\mathcal{J}_U/U).$$
 then $\mathcal{J}_{U_1}^{\circ}$ identifies to ${\rm Pic}^0(\mathcal{J}_{U_1/U_1})$ and by uniqueness,
 the compactification $\mathcal{J}_{U_1}$ identifies with the N\'eron compactification
 of ${\rm Pic}^0(\mathcal{J}_{U_1/U_1})$ parameterizing rank $1$ torsion free sheaves
 on fibers of $\pi_{U_1}:\mathcal{J}_{U_1}\rightarrow U_1$.
 The action over $U_1$ of $\mathcal{J}_{U_1}^{\circ}$ on $\mathcal{J}_{U_1}$ is then the action
 of ${\rm Pic}^0$ on sheaves by
tensor product.
 \end{proof}
 Formula (\ref{eqformulaju10}) shows that $R^3u_{1*}(\mathbb{Z}/3\mathbb{Z})\subset \mathcal{J}_{U_1}^{\circ}$ identifies with the
 sheaf $_3\mathcal{J}_{U_1}^{\circ}$ of $3$-torsion points and Lemma \ref{leaction} shows that
 $\mathcal{J}_{U_1}^{\circ}\subset \mathcal{A}ut\,(\mathcal{J}_{U_1}/U_1)$.
 Let us now exhibit the twisting class in $H^1_{et}(U_1,\mathcal{A}ut\,(\mathcal{J}_{U_1}/U_1))$ needed to construct $\mathcal{J}_{U_1}^{T}$. The twisting class will be of $3$-torsion
 and more precisely will come
 from a class in $H^1_{et}(U_1,\,\,_3\mathcal{J}_{U_1}^{\circ})=H^1(U_1,\,\,_3\mathcal{J}_{U_1}^{\circ})$, where in the right hand side we consider  cohomology with respect to the usual
 topology.
 Consider the morphism
 $u_1:\mathcal{Y}_{u_1}\rightarrow U_1$.
 As they admit at worst one node, the fibers of $u_1$ satisfy $H^4(\mathcal{Y}_t,\mathbb{Z})=\mathbb{Z}$ with
generator at the point $t$ the class of a line (not passing through the node of
$\mathcal{Y}_t$ if $\mathcal{Y}_t$ is singular). Let $h=c_1(\mathcal{O}_X(1)$ and $h_{\mathcal{Y}}\in H^2(\mathcal{Y},\mathbb{Z})$ be its pull-back to $\mathcal{Y}$.
The image $\widetilde{h_\mathcal{Y}^2}$ of the class $h_\mathcal{Y}^2$ in $H^0(U_1,R^4u_{1*}\mathbb{Z})$ is equal
to $3$ times the above generator of  $R^4u_{1*}\mathbb{Z}=\mathbb{Z}$. The twisting class
comes from the lack of degeneracy at $E_2$ of the Leray spectral sequence of $u_1$ with
integral coefficients.

 \begin{lemm}\label{lepourt} Let $\sigma\in H^0(U_1,R^4u_{1*}\mathbb{Z})$ be the natural generator. Then
 the image $d_2\sigma\in H^2(U_1,R^3u_{1*}\mathbb{Z})$ is of $3$-torsion, and comes from a
 canonically defined  element $$t\in H^1(U_1,R^3u_{1*}(\mathbb{Z}/3\mathbb{Z})).$$
 \end{lemm}
 \begin{proof} The class $d_2\sigma\in H^2(U_1,R^3u_{1*}\mathbb{Z})$ is of $3$-torsion because $3\sigma=\widetilde{h_\mathcal{Y}^2}$ is the image of $h_{\mathcal{Y}}^2\in H^4(\mathcal{Y}_{U_1},\mathbb{Z})$ in $H^0(U_1,R^4u_{1*}\mathbb{Z})$.
 The exact sequence
  $$0\rightarrow \mathbb{Z}\stackrel{3}{\rightarrow}\mathbb{Z}\rightarrow\mathbb{Z}/3\mathbb{Z}
 \rightarrow0$$
 of constant
 sheaves  on $\mathcal{Y}_{U_1}$
 induces the exact sequence
 \begin{eqnarray}\label{eqexlongue} 0\rightarrow R^3u_{1*}\mathbb{Z}\stackrel{3}{\rightarrow}R^3u_{1*}\mathbb{Z}\rightarrow
 R^3u_{1*}(\mathbb{Z}/3\mathbb{Z})\rightarrow 0
 \end{eqnarray}
 because $R^3u_{1*}\mathbb{Z}$ and $R^4u_{1*}\mathbb{Z}$ have no torsion.
 The long exact sequence associated to (\ref{eqexlongue}) thus shows that
 $d_2\sigma$ lifts to an element of
 $H^1(U_1,R^3u_{1*}(\mathbb{Z}/3\mathbb{Z}))=H^1(U_1,\,_3\mathcal{J}_{U_1}^\circ)$.
This does not prove however that this lift is canonical since  in the considered long exact sequence of Betti cohomology, there
is a nontrivial cohomology group
$H^1(U_1,R^3u_{1*}\mathbb{Z})$ (this is due to the non-triviality of
$H^4(X,\mathbb{Z})$). We can however go around this difficulty
by considering the universal situation where instead of considering
the universal family $\mathcal{Y}_{U_1}$ of smooth or
$1$-nodal  hyperplane sections of the given cubic fourfold $X$, we consider
the universal family $v_{1}:\mathcal{Y}^{univ}_{W_1}\rightarrow W_1$ of smooth or
$1$-nodal cubic threefolds contained in some hyperplane   $\mathbb{P}^4\subset \mathbb{P}^5$.
Here the whole parameter space $B_{univ}$ is the projective bundle
over $(\mathbb{P}^5)^\vee$ with fiber over $[H]$ the projective space
$\mathbb{P}(H^0(\mathcal{O}_H(3)))$ and $W_1\subset B_{univ}$ is a Zariski open set in it.
In this case, we have
\begin{sublemma} One has $H^1(W_1, R^3v_{1*}\mathbb{Z})=0$.
\end{sublemma}
\begin{proof} One proves the result first of all with $\mathbb{Q}$-coefficients, using the fact that
the Leray spectral sequence degenerates in $E_2$ and that the degree $4$ cohomology of
$\mathcal{Y}^{univ}_{W_1}$ is algebraic and very simple.  Then one concludes using the exact sequence
on $W_1$
\begin{eqnarray}\label{eqexlonguetors} 0\rightarrow R^3v_{1*}\mathbb{Z}{\rightarrow}R^3v_{1*}\mathbb{Q}\rightarrow
 R^3v_{1*}(\mathbb{Q}/\mathbb{Z})\rightarrow 0
 \end{eqnarray}
 and the fact that $H^0(W_1,R^3v_{1*}(\mathbb{Q}/\mathbb{Z}))=0$.
\end{proof}
This fact gives us a canonical  lift of $d_2\sigma^{univ}$ to an element $t^{univ}$ of
 $H^1(W_1,R^3v_{1*}(\mathbb{Z}/3\mathbb{Z}))$ hence by pull-back to
 $U_1$ (using the natural map $g:U_1\rightarrow W_1$), the desired  canonical  lift $t$ of $d_2\sigma$ to an element of
 $H^1(U_1,R^3u_{1*}(\mathbb{Z}/3\mathbb{Z}))$.
  \end{proof}
Lemma \ref{lepourt} gives us a class in $H^1(U_1,\,_3\mathcal{J}_{U_1})=H^1_{et}(U_1,\,_3\mathcal{J}_{U_1})$ which by Lemma
\ref{leaction} allows us to construct a twisted family
$\pi_{U_1}:\mathcal{J}_{U_1}^T\rightarrow U_1$, which is above
$U$  a torsor over the group scheme $\mathcal{J}_U$.
Note that the proof of Lemma \ref{lepourt} also shows that $\mathcal{J}_{U_1}^T$ is the pull-back
under $g_1:U_1\rightarrow W_1$ of the corresponding object $\mathcal{J}_{W_1}^{T,univ}$ over
$W_1$.
We now prove that the torsor so constructed is actually the object we want, namely, the target
of the Abel-Jacobi map for $1$-cycles of degree $1$ along the smooth fibers of $u$.
This will follow from the following result: Let as before $W_1\subset B_{univ}$ be the Zariski open set  parameterizing  cubic threefolds
in $\mathbb{P}^5$ with at most one ordinary double point. Let $ o_1: \mathcal{F}_1^{univ}\rightarrow W_1$ be the family of lines in the fibers of
$v_1:\mathcal{Y}_{W_1}^{univ}\rightarrow W_1$ and $\mathcal{F}_1^{0,univ}\subset \mathcal{F}_1^{univ}$ be the Zariski open set
consisting of lines not passing through the node when the corresponding cubic threefold
is singular. Let $o_{1,0}$ be the restriction of $o_1$ to $\mathcal{F}_1^0$.
\begin{lemm} \label{letrivdroite1}
There exists an isomorphism
\begin{eqnarray}
\label{eqisodu14} o_{1,0}^*\mathcal{J}_{W_1}^{univ}\cong o_{1,0}^*\mathcal{J}_{W_1}^{T,univ}
\end{eqnarray}
of  quasi-projective varieties over
$\mathcal{F}_1^{0,univ}$.
\end{lemm}

\begin{proof}
It suffices  to show that the pull-back $o_{1,0}^*(t)$ vanishes in
$H^1(\mathcal{F}_{1}^{0,univ},o_{1,0}^*(_3\mathcal{J}_{W_1}^{univ}))$.
Note that the morphism $o_{1,0}$ is smooth, because
the family of lines in the smooth locus of a cubic hypersurface is smooth. It follows
 that the total space of the  universal family
$$\mathcal{Y}^{univ}_{\mathcal{F}_1^{0,univ}}:=
\mathcal{Y}^{univ}_{W_1}\times_{W_1}\mathcal{F}_{1}^{0,univ}$$
is smooth. Let $v'_1:\mathcal{Y}^{univ}_{\mathcal{F}_1^{0,univ}}\rightarrow \mathcal{F}_1^{0,univ}$ be the natural morphism.
 Let $H^4_{\mathbb{Z}},\,H^3_{\mathbb{Z}},\,H^3_{\mathbb{Z}/3\mathbb{Z}}$ be respectively the sheaves
 $$R^4v'_{1*} \mathbb{Z}=o_{1,0}^*(R^4v_{1*} \mathbb{Z}),\,\,  R^3v'_{1*} \mathbb{Z}=o_{1,0}^*(R^3v_{1*} \mathbb{Z}), \,\, R^3v'_{1*} (\mathbb{Z}/3\mathbb{Z})=o_{1,0}^*(R^3v_{1*}(\mathbb{Z}/3\mathbb{Z}))$$ on
$\mathcal{F}_1^{0,univ}$.
The class $$o_{1,0}^*\sigma\in H^0(\mathcal{F}_1^{0,univ},o_{1,0}^*(R^4v_{1*} \mathbb{Z}))=H^0(\mathcal{F}_1^{0,univ},H^4_{\mathbb{Z}})$$ comes from a class in $H^4(\mathcal{Y}^{univ}_{\mathcal{F}_1^{0,univ}},\mathbb{Z})$, namely the class $[\Delta^{univ}]$ of the universal line
$\Delta^{univ}\subset \mathcal{Y}^{univ}_{\mathcal{F}_{1}^{0,univ}}$.
 We thus get that $d_2(o_{1,0}^*\sigma)=0$ in $H^2(\mathcal{F}_1^{0,univ},H^3_{\mathbb{Z}})$.
  This means that  the class
$$o_{1,0}^*(t)\in H^1(\mathcal{F}_1^{0,univ}, H^3_{\mathbb{Z}/3\mathbb{Z}})=
H^1(\mathcal{F}_1^{0,univ},\,\,_3\mathcal{J}^{univ}_{\mathcal{F}_1^0})$$ vanishes in
$H^2(\mathcal{F}_1^{0,univ}, H^3_{\mathbb{Z}})$ (using as before the long exact sequence
associated to the  short exact sequence
$$0\rightarrow H^3_{\mathbb{Z}}\rightarrow H^3_{\mathbb{Z}}\rightarrow H^3_{\mathbb{Z}/3\mathbb{Z}}\rightarrow0.$$
In order to prove that $o_{1,0}^*(t)$ vanishes in $H^1(\mathcal{F}_1^{0,univ},\,\,_3\mathcal{J}^{univ}_{\mathcal{F}_1^0})$, it  thus suffices to show that $H^1(\mathcal{F}_W, o^*H^3_{\mathbb{Z}})=0$, which is done exactly as before.
\end{proof}
\begin{rema}{\rm The isomorphism between the pullbacks to
$\mathcal{F}_1^{0,univ}$ of the twisted and the untwisted families
given in Lemma \ref{letrivdroite1} is not canonical since over $\mathcal{F}_W^{univ}$ the
intermediate Jacobian fibration has a nonzero section. Indeed, for each cubic smooth cubic threefold
$Y$ with a line $\Delta\subset Y$, the $1$-cycle
$3\Delta-h^2$ is cohomologous to zero on $Y$, hence has a  nontrivial Abel-Jacobi invariant.
This provides a nontrivial section which acts by translation on $o^*\mathcal{J}_{W}^{univ}$ over
$\mathcal{F}_W$. One can show that the isomorphism is unique up to the action of the group generated by this  translation.}
\end{rema}
\begin{proof}[Proof of Proposition \ref{proptwist}] We already constructed $\mathcal{J}_{U_1}^T$ as
the pull-back via $g_1:U_1\rightarrow W_1$ of the twisted family
$\mathcal{J}_{W_1}^T$. That it is \'etale locally isomorphic to $\mathcal{J}_{U_1}$ over $U_1$ is by construction. The only thing to prove is point (2).
However, Lemma \ref{letrivdroite1} shows that for any base change
morphism $f: U'\rightarrow U$,  assuming  there is
 a family of lines  $\Delta_{U'}\subset \mathcal{Y}_{U'}$, there is a canonical morphism
 $\Phi_{\Delta}:U'\rightarrow \mathcal{J}_{U}^T$ over $U$. Indeed, the data above give a morphism
 $c: U'\rightarrow \mathcal{F}_W^{univ}$ so we can use $\Phi_{\Delta}^{univ}$ of the previous lemma, and define $\Phi_\Delta:=\Phi_{\Delta^{univ}}\circ c$.
 To see that a family of $1$-cycles of degree $1$ in the fibers of
 $\mathcal{Y}_{U'}\rightarrow U'$ suffices, we can use the universal generation
 result of \cite{shen} which says that
 any $1$-cycle on a smooth  cubic hypersurface $Y$ over a field $K$ comes from a $0$-cycle
 on the surface of lines $F(Y)$, also defined over $K$. In our case, $K$ is the function field of
 $U'$. Using this result,  we get a correspondence over $U'$ between
 $U'$ and $\mathcal{F}_{U'}$, which has to be of degree $1$. Using this correspondence and the previous step, we construct the section $U'\rightarrow \mathcal{J}_{U'}^T$.
 \end{proof}
\section{Descent \label{sec4}}
Our goal in this section is to explain how to mimic the arguments from
\cite{lazasaccavoisin} in order to construct for a general cubic fourfold
 $X$ a hyper-K\"ahler compactification of the twisted
intermediate Jacobian fibration
$\mathcal{J}_{U_1}^T$ constructed in the previous section.
Recall the notion of very good line in a cubic threefold $Y$, introduced and used in \cite{lazasaccavoisin} (see also \cite{CML} for  a  slightly weaker notion): A line $\Delta\subset Y$ is very good if $\Delta$ does not pass through
the singular point of $Y$ (equivalently the surface of lines $F(Y)$ is smooth
at $[\Delta]$), the
curve $\widetilde{C}_\Delta$ of lines in $Y$  meeting $\Delta$ is irreducible and the natural involution acting on $\widetilde{C}_\Delta$ has no fixed points.
One of the results proved in \cite{lazasaccavoisin} (improving previous
results of \cite{CML}) is:
\begin{prop} \label{prolsv} If $X$ is a general cubic $4$-fold, and $Y\subset X$ is any hyperplane
section, $Y$ has a very good line.
\end{prop}
Let
$o_{vg}:\mathcal{F}^{vg}\rightarrow B$ be the family of very good lines in the fibers
of $u:\mathcal{Y}\rightarrow B$. Proposition \ref{prolsv}
says that this morphism is surjective and by definition it is smooth. Restricting over $U_1$, we get a Zariski open set $$\mathcal{F}_1^{vg}\subset \mathcal{F}_1^{0},\,\,o_{vg,1}:\mathcal{F}_1^{vg}\rightarrow U_1$$ with its pulled-back intermediate
Jacobian fibration
$o_{vg,1}^*\mathcal{J}_{U_1}:=\mathcal{F}_1^{vg}\times_{U_1}\mathcal{J}_{U_1}$, which is also isomorphic to $o_{vg,1}^*\mathcal{J}_{U_1}^T$
by Lemma \ref{letrivdroite1}.
In \cite{lazasaccavoisin}, a smooth quasiprojective  compactification
$\pi_{\mathcal{F}^{vg}}:\overline{\mathcal{P}}_{\mathcal{F}^{vg}}\rightarrow \mathcal{F}^{vg}$ of $o_{vg,1}^*\mathcal{J}_{U_1}\rightarrow \mathcal{F}^{vg}_1$
is constructed, with a morphism $\pi_{\mathcal{F}^{vg}}$ which is
flat and projective.
The compactified intermediate Jacobian fibration $\overline{\mathcal{J}}$ is then obtained
by descending the fibration $\pi_{\mathcal{F}^{vg}}$.
The statement which makes it possible is:
\begin{lemm} \label{lerelample}There exists a line bundle $\mathcal{O}(\Theta_1)$ on $\mathcal{J}_{U_1}$ whose pull-back $o_{vg,1}^*\mathcal{O}(\Theta_1)$
to $o_{vg,1}^*\mathcal{J}_{U_1}$ extends uniquely to a relatively ample line
bundle $\mathcal{O}(\Theta^{vg})$ on $\overline{\mathcal{P}}_{\mathcal{F}^{vg}}$.
\end{lemm}

Note that the extension of the pull-back $o_{vg,1}^*\mathcal{O}(\Theta_1)$  to a line
bundle on $\overline{\mathcal{P}}_{\mathcal{F}^{vg}}$ exists and is unique because $\overline{\mathcal{P}}_{\mathcal{F}^{vg}}$ is smooth and $o_{vg,1}^*\mathcal{J}_{U_1}\subset \overline{\mathcal{P}}_{\mathcal{F}^{vg}}$ is Zariski open with complement of codimension at least
$2$. This last  point follows from the  flatness of $\pi_{\mathcal{F}^{vg}}$   because
$\mathcal{F}_1^{vg}\subset \mathcal{F}^{vg}$ is Zariski open with
complement of codimension $\geq 2$. The important point in Lemma
\ref{lerelample} is thus relative ampleness. Once we have the line bundle
$\mathcal{O}(\Theta_1)$ on $\mathcal{J}_{U_1}$ as in the lemma, we get
the formula defining $\overline{\mathcal{J}}$ as a Proj over
$B$, namely, letting $j_1:U_1\rightarrow B$ be the natural inclusion,
we set
\begin{eqnarray}\label{eqpourJbar}\overline{\mathcal{J}}=Proj\,(\oplus_k j_{1*}(R^0\pi_{U_1*}\mathcal{O}(k\Theta_1)).
\end{eqnarray}
Using the fact that $B\setminus U_1$ has codimension $2$ in $B$, flatness of $\pi_{\mathcal{F}^{vg}}$
and relative ampleness of (the extension of) the pull-back of $\Theta_1$, we
see that  the sheaf of algebras
$\oplus_k j_{1*}(R^0\pi_{U_1*}\mathcal{O}(k\Theta_1))$ is a sheaf of finitely generated algebras over
$\mathcal{O}_B$ whose graded pieces are locally free and that the Proj is smooth, because
all these properties become true after pull-back to
$\mathcal{F}^{vg}$. Indeed,  the sheaf of algebras $\oplus_k j_{1*}(R^0\pi_{U_1*}\mathcal{O}(k\Theta_1))$ pulls-back to
$\oplus_k R^0\pi_{\mathcal{F}^{vg}*}\mathcal{O}(k\Theta^{vg}))$ on $\mathcal{F}^{vg}$. (To be completely rigorous here, we should in fact replace the line bundle $\Theta_1$ by a multiple.)

The way this descent construction works also makes clear what we need in order to treat the twisted case. Note that what we are going to do is to descend
the same fibration $\pi_{\mathcal{F}^{vg}}:\overline{\mathcal{P}}_{\mathcal{F}^{vg}}\rightarrow \mathcal{F}^{vg}$
to a flat fibration $\pi^T:\overline{\mathcal{J}}^T\rightarrow B$
with a different descent data, given  by a different relatively ample divisor.
More precisely, consider  $\pi_{\mathcal{F}^{vg}}:\overline{\mathcal{P}}_{\mathcal{F}^{vg}}
\rightarrow \mathcal{F}^{vg}$ as a relative flat projective compactification of $o_{vg,1}^*\mathcal{J}^T_{U_1}\rightarrow \mathcal{F}_1^{vg}$ using Lemma
\ref{letrivdroite1}. The only ingredient needed is the following:
\begin{prop} \label{prothetaT} There exists a  line bundle $\mathcal{O}(\Theta^{vg,T})$ on $\overline{\mathcal{P}}_{\mathcal{F}^{vg}}$, which is relatively ample over $\mathcal{F}^{vg}$, and whose restriction to
$\mathcal{F}_1^{vg}$
is the pull-back (using the isomorphism (\ref{eqisodu14})) of a line bundle $\mathcal{O}(\Theta_1^T)$ on
$\mathcal{J}_{U_1}^T$.
\end{prop}
Indeed, once one has the line bundles $\mathcal{O}(\Theta^{vg,T})$ and $\mathcal{O}(\Theta_1^T)$ as above, one defines $\overline{\mathcal{J}}^T$ as
\begin{eqnarray}\label{eqpourJbarT}\overline{\mathcal{J}}^T=Proj\,(\oplus_k j_{1*}(R^0\pi_{U_1*}\mathcal{O}(k\Theta_1^T))
\end{eqnarray}
and the same arguments as above show that this is a smooth projective variety, flat over
$B$ and extending $\mathcal{J}^T_{U_1}$.
\begin{proof}[Proof of Proposition \ref{prothetaT}] It is proved in \cite[Section 4]{lazasaccavoisin} that the fibers of $\pi_{\mathcal{F}_{vg}}: \overline{\mathcal{P}}_{\mathcal{F}^{vg}}\rightarrow \mathcal{F}^{vg}$ are irreducible. We now
use the following lemma:
\begin{lemm}\label{lepournumtri} Let $M$ be a smooth irreducible quasiprojective variety, $f:M\rightarrow N$ be a flat  projective morphism with irreducible fibers $M_t$, $\forall t\in N$, and
let $\mathcal{L}\in{\rm Pic}\,M$. If $\mathcal{L}_{\mid M_t}$  is topologically trivial for
the  general point $t\in N$, $\mathcal{L}_{\mid M_t}$  is numerically trivial for all
$t\in N$.
\end{lemm}
\begin{proof} The statement is local on $N$. Let $t_0\in N$ and let $C\subset M_{t_0}$ be a curve.
Choosing a sufficiently relatively ample line bundle
$H$ on $M$ and using the fact that
$M_{t_0}$ is irreducible, we can assume that $C$ is contained in an irreducible surface $S_{t_0}\subset M_{t_0}$ which is a
complete intersection of members of $|H_{\mid M_{t_0}}|$.
Up to replacing $H$ by a multiple and shrinking $N$ if necessary, we can construct
a  flat family $f_S: S\rightarrow N,\,S\subset M$ of complete intersection surfaces with fiber
$S_{t_0}$ over the point $t_0$. As the line bundle $\mathcal{L}_{\mid M_t}$ is topologically trivial for general $t$, we conclude that
$$c_1(\mathcal{L})\cdot c_1(H)\cdot  S_t=0,\,\,c_1(\mathcal{L})^2\cdot  S_t=0$$
for general $t$, and by flatness, this is also true for $S_{t_0}$.
Let $\tau:\widetilde{S_{t_0}}\rightarrow S_{t_0}$
be a desingularization.
The line bundle
$\mathcal{L}'_0:=\tau^*\mathcal{L}_{\mid S_{t_0}}$ on the smooth connected  surface
$\widetilde{S_{t_0}}$ satisfies
$$c_1(\mathcal{L}'_0)^2=0,\,c_1(\mathcal{L}'_0)\cdot c_1(\tau^*H)=0.$$
As $ c_1(\tau^*H)^2>0$, it follows from the Hodge index theorem that $\mathcal{L}'_0$
is topologically trivial modulo torsion on $\widetilde{S_{t_0}}$. In particular, if $C'\subset
\widetilde{S_{t_0}}$ is a curve mapping onto $C\subset S_{t_0}$, we have
${\rm deg}\,\mathcal{L}'_{0\mid C'}=0= D {\rm deg}\,\mathcal{L}_{\mid C}$, where
$D$ is the degree of $C'$ over $C$. Thus ${\rm deg}\,\mathcal{L}_{\mid C}=0$.
\end{proof}
Let now $\Theta_1^T$ be any relatively ample line bundle on
$\mathcal{J}_{U_1}^T\rightarrow U_1$. Its pull-back to
$\mathcal{J}^T_{\mathcal{F}_1^{vg}}\cong \mathcal{J}_{\mathcal{F}_1^{vg}}$ extends to a line bundle
$\mathcal{O}(\Theta^{vg,T})$ on the compactified family of Prym varieties $\pi_{\mathcal{F}^{vg}}:\overline{\mathcal{P}}_{\mathcal{F}^{vg}}\rightarrow \mathcal{F}^{vg}$.
We  also have on $\mathcal{J}_{\mathcal{F}_1^{vg}}$  the pull-back of the relatively ample line bundle
$\Theta_1$  on
$\mathcal{J}_{U_1}\rightarrow U_1$. The later extends (in fact uniquely) to a relatively
ample line bundle $\mathcal{O}(\Theta^{vg})$ on $\overline{\mathcal{P}}_{\mathcal{F}^{vg}}\rightarrow \mathcal{F}^{vg}$. Next, an easy monodromy argument shows
(see \cite[Section 5]{lazasaccavoisin}) that the N\'eron-Severi group of the intermediate Jacobian of a very general
cubic $3$-fold is isomorphic to
$\mathbb{Z}$. It follows that for adequate positive integers $a,\,b$,
the line bundle
$\mathcal{O}(a\Theta^{vg,T}-b \Theta^{vg})$ is topologically trivial on the general fibers
 of $\pi_{\mathcal{F}^{vg}}$. We then conclude by Lemma
 \ref{lepournumtri} that $\mathcal{O}(a\Theta^{vg,T}-b \Theta^{vg})$ is numerically trivial on all the fibers
 of $\pi_{\mathcal{F}^{vg}}$. Thus $\mathcal{O}(a\Theta^{vg,T})$ is the sum of an ample line bundle and a numerically
 trivial line bundle on any fiber of $\pi_{\mathcal{F}^{vg}}$, hence it is relatively ample.
\end{proof}
This concludes the construction of the smooth  projective variety
$\overline{\mathcal{J}}^T$. The fact that this is a hyper-K\"ahler manifold follows easily, as in
\cite{lazasaccavoisin}:
The existence of the holomorphic $2$-form works as in \cite{lazasaccavoisin}. The fact that it is nondegenerate is due to the fact that it is nondegenerate on $\mathcal{J}^T_{U_1}$ since
$B\setminus U_1$ has codimension $\geq 2$ in $B$ and $\pi:\overline{\mathcal{J}}^T\rightarrow B$ is flat. The fact that the holomorphic $2$-form is nondegenerate on $\mathcal{J}^T_{U_1}$ follows
from the similar statement for the untwisted family since they are \'etale locally isomorphic.
Finally, the fact that the variety we construct is actually hyper-K\"ahler (i.e. simply connected with only one holomorphic $2$-form up to a coefficient) follows from the fact that the two varieties
are birational, hence deformation equivalent, when  the cubic fourfold $X$  acquires an integral Hodge class $\alpha\in H^4(X,\mathbb{Z})$ which has degree $1$ on its hyperplane sections. Note that
the above construction of $\overline{\mathcal{J}}^T$ and $\overline{\mathcal{J}}$ works a priori only for general $X$, but the set of special $X$'s as above is Zariski dense (this is a countable union
of hypersurfaces, dense for the usual topology in the moduli space of cubic fourfolds), hence
there are points which correspond to special $X$'s in Hassett's sense (see \cite{hassett}), with
a special class of degree $1$ along hyperplane sections, and for which the constructions
of $\overline{\mathcal{J}}^T$ and $\overline{\mathcal{J}}$ specialize well.
\begin{rema}{\rm  The two varieties $\overline{\mathcal{J}}^T$ and $\overline{\mathcal{J}}$
are isogenous in the sense that there is a
rational map (of degree $3^{10}$) $\overline{\mathcal{J}}^T\dashrightarrow \overline{\mathcal{J}}$. This is obvious from the fact that the open part $\mathcal{J}^T_U$ is constructed as a torsor over the group scheme $\mathcal{J}^T_U$, with a twisting class
of order $3$. We believe but did not prove that the two varieties are not birational. We only note that by construction  they are  not birational over $B$.}
\end{rema}

 Coll\`ege de France

3  rue d'Ulm, 75005 Paris

 France

\smallskip
 claire.voisin@imj-prg.fr
    \end{document}